\documentclass[12pt,thmsa, reqno]{amsart}

\usepackage[dvips]{graphics}
\input{amssym.def}
\input{amssym.tex}

\usepackage{lineno}

\def\Cal{\mathcal}

\def\T{{\Cal T}}

\def\Z{\mathcal{Z}}

\def\stnm{{\rm St}(n,m)}

\def\stnk{{\rm St}(n\!+\!1,n\!+\!1\!-\!k)}

\def\sn{\bbs^n}

\def\bbr{{\Bbb R}}

\def\bbc{{\Bbb C}}

\def\bbs{{\Bbb S}}
\def\bbb{{\Bbb B}}

\def\ker{{\hbox{\rm ker}}}

\def\det{{\hbox{\rm det}}}

\def\rn{\bbr^n}

\def\part{\partial}
\def\intl{\int\limits}
\def\b{\beta}

\def\a{\alpha}
\def\om{\omega}

\def\del{\delta}
\def\vp{\varphi}

\def\gam{\gamma}

\def\sig{\sigma}
\def\lam{\lambda}
\def\z{\zeta}
\def\th{\theta}
\def\e{\varepsilon}
\def\t{\tau}

\def\chi{{\bf 1}}

\def\snm1{\bbs^{n-1}}

\font\frak=eufm10

\def\fr#1{\hbox{\frak #1}}

\def\frM{\fr{M}}

\newtheorem{theorem}{Theorem}[section]
\newtheorem{lemma}[theorem]{Lemma}

\theoremstyle{definition}

\theoremstyle{remark}
\newtheorem{remark}[theorem]{Remark}

\theoremstyle{corollary}

\numberwithin{equation}{section}

\newcommand{\be}{\begin{equation}}
\newcommand{\ee}{\end{equation}}

\newcommand{\bea}{\begin{eqnarray}}
\newcommand{\eea}{\end{eqnarray}}
\newcommand{\Bea}{\begin{eqnarray*}}
\newcommand{\Eea}{\end{eqnarray*}}

\def\sideremark#1{\ifvmode\leavevmode\fi\vadjust{\vbox to0pt{\vss
 \hbox to 0pt{\hskip\hsize\hskip1em
\vbox{\hsize2cm\tiny\raggedright\pretolerance10000
 \noindent #1\hfill}\hss}\vbox to8pt{\vfil}\vss}}}%

\begin{document}

\title[Funk-Type Transforms]
{On Two Families of Funk-Type Transforms}

\author{M. Agranovsky and B. Rubin}

\address{Department of Mathematics, Bar Ilan University, Ramat-Gan, 5290002, and Holon Institute of Technology, Holon, 5810201, Israel}
\email{agranovs@math.biu.ac.il }

\address{Department of Mathematics, Louisiana State University, Baton Rouge,
Louisiana 70803, USA}
\email{borisr@math.lsu.edu}

\subjclass[2010]{Primary 44A12; Secondary 37E30}


\maketitle

\begin{abstract}
 We consider two families of Funk-type transforms that assign to a function on the unit sphere the integrals of that function over spherical sections by planes of fixed dimension.  Transforms of the first kind are generated by planes passing through a fixed center  outside the sphere. Similar transforms with interior center and with center on the sphere itself we studied in previous publications. Transforms of the second kind, or the parallel slice transforms, correspond to
   planes that are parallel to a fixed direction.   We show that the Funk-type transforms with exterior center express through the parallel slice transforms and the latter are intimately related to the Radon-John d-plane transforms on the Euclidean ball.  These results allow us to investigate  injectivity of our  transforms and obtain  inversion formulas for them.  We also establish connection between the Funk-type transforms  of  different dimensions with  arbitrary center.

 \end{abstract}

\section {Introduction}

The present investigation deals with a pair of integral operators on the unit sphere $\sn$ in $\bbr^{n+1}$ that arise in spherical tomography \cite{GRS} and resemble the classical Funk transform. These operators have the form
\be\label {mnv8} (F_a f)(\t) = \intl_{\sn \cap \,\t} f(x)\, d\sig (x), \qquad (\Pi_{a} f)(\z)=\intl_{\sn \cap \,\z} f(x)\, d\sig (x).\ee
In the first integral, a  function $f$ is integrated over the spherical section $\sn \cap \,\t$
by the $k$-dimensional plane $\t$ passing through a fixed point  $a \in \bbr^{n+1}$, $1< k\le n$. In the second integral, the spherical section  $\sn \cap \,\z$
 is determined by the $k$-plane $\z$, which is parallel to the vector $a\ne 0$.
 We call these operators the {\it shifted Funk transform} and the {\it spherical slice transform}, respectively.
 Our main concern  is  injectivity of the operators (\ref{mnv8}) and inversion formulas.

The case  when $a=0$ is the origin of $\bbr^{n+1}$ corresponds  to the  totally geodesic transform $F_0$, which  goes back to    Minkowski \cite{Min}, Funk  \cite{Fu11, Fu13},
  and Helgason \cite{H90}; see also \cite {GGG1, H11,  Ru02b, Ru15} and references therein. This operator annihilates odd functions and can be  explicitly inverted on even functions.

  The case $|a|=1$, when $F_a$ is injective,   was considered
 by Abouelaz and  Daher \cite{AbD} (on zonal functions)  and  Helgason \cite [p. 145] {H11} (on arbitrary functions) for $n=2$. The general case of all $n\ge 2$ for hyperplane sections was studied by the second co-author
 \cite[Section 7.2]{Ru15}. The method of this work extends to all $1< k\le n$ and inversion formulas can be obtained by making use of the stereographic projection and the inversion results for the Radon-John $d$-plane  transforms.

Operators  $F_a$ with $|a|<1$ were studied  by Salman \cite{Sa16, Sa17} (in a  different setting, caused by support restrictions) and later by  Quellmalz \cite{Q, Q1} and Rubin \cite{Ru18} for $k=n$. The case of  all  $1< k\le n$ was considered by Agranovsky and Rubin \cite {AR}. The kernel of such operators consists of functions $f$ that are odd in a certain  sense, whereas  the corresponding even components  of $f$  can be explicitly reconstructed from $F_af$.

We give credit to  related works by Gardner \cite {Ga}, Gindikin,  Reeds, and   Shepp \cite{GRS}, Kazantsev \cite {Ka},  Palamodov \cite{P1, P2}, and many others, containing information on diverse operators  of integral geometry on the unit sphere and extensive references on this subject.

The present article is focused on the unknown  case $|a|>1$ of the exterior center. We shall prove that,  unlike the case $|a|<1$, when $F_a$ is equivalent, in a certain sense,  to the Funk transform $F_0$, now it  is similarly related to the  parallel slice transform  $\Pi_a$.
 When $a$ belongs to the last coordinate axis, the operator $\Pi_a$  is  known as the {\it vertical slice  transform}. It  was studied
by  Hielscher and  Quellmalz \cite {HQ} (for $n=2$) and Rubin \cite {Ru18a} (for $k=n\ge 2$). In the present article we consider  $\Pi_a$  for all $k\le n$ and all $a\neq 0$.

{\bf Plan of the Paper and Main Results.}  Section 2 contains preliminaries,
 including notation and   auxiliary facts about the  Radon-John $d$-plane transforms.

Section 3 deals with the parallel slice transform $\Pi_a$.   We show that $\Pi_a$ reduces to the Radon-John transform of functions on the  unit ball in the subspace $a^\perp$ perpendicular to $a$. Hence  $\Pi_a$ can be explicitly inverted on the subspace of continuous functions that are even with respect to reflection in $a^\perp$.

Section 4 contains  auxiliary facts about  M\"obius automorphisms $\vp_{a}$   and  $\vp_{a^*}$, $a^*=a/|a|^2$ (see  (\ref{E:g_a})) that  play an important role in the sequel. The definition of these  automorphisms is borrowed from the book by Rudin \cite[Section 2.2.1]{Rud}.

Section 5  contains main results of the paper related to  injectivity and inversion of the shifted Funk transform $F_a$, $|a|>1$,  on continuous functions. The consideration relies on the equality
\be\label {ajhgw} F_a f =  (\Pi_a M_{a^*} f)\circ \vp_{a^*} ,\ee
which links together  $F_a$ and $\Pi_a$ by means of $\vp_{a^*}$ and the  bijection $ M_{a^*}$ having the form
\be\label{vts1de}
(M_{a^*}f)(y)\!=\!\left ( \frac {s_{a^*} }{1\!-\!a^*\cdot  y}\right )^{k-1}\!\!(f\circ \vp_{a^*})(y), \quad s_{a^*}\!=\!\sqrt{1\!-\!|a^*|^2}; \ee
 see  Theorem \ref{ukdndt}.  This  equality is  the core of the paper. It paves the way to the study of injectivity and inversion of $F_a$ (see Theorems \ref{luqw}, \ref{i00s}).
  In the case $|a|<1$, a similar formula, but with the geodesic Funk transform $F_0$ in place of $\Pi_a$, was obtained by Quellmalz \cite {Q} for $n=2$ and generalized  in  \cite{AR, Q1, Ru18} to higher dimensions. In order not to  overload Section 5 with technicalities, we move the proof of (\ref{ajhgw}) to Section 6.

Section 7 can be considered as an  addendum to the previous sections and can be read independently. Here the main result is Theorem \ref{mgvdr} that
 establishes connection between the  shifted Funk transforms $F_a$ associated with
planes of different dimensions. The point $a$ in this theorem is   arbitrary.
   Theorem \ref{mgvdr}   reduces inversion of the  lower dimensional  transforms  to those with  $k=n$. Although the case $k=n$ is less technical, the payment for this simplicity is
   that the inversion procedure contains  an additional  integration step and the degree of the differential operator in the corresponding Radon-John inversion component becomes bigger, specifically,  $n-1$ versus $k-1$. We think that exposition of both methods is, perhaps, instructive.

In our previous work \cite{AR}, we also studied the paired shifted Funk transform $f \to \{F_a f, F_b f\}$ with two  distinct  centers $a$ and $b$ inside the sphere. This approach makes it possible  to circumvent non-injectvity of the single-centered transform $F_a$. Under certain conditions, the results of \cite{AR} extend to {\it arbitrary} two centers in the space, e.g., one center  inside $\sn$ and another outside, or both  are outside. We plan to write a separate paper on  this subject.

\section {Preliminaries}

\subsection {Notation}  In the following,  $\bbr^{n+1}$ is the real $(n+1)$-dimensional space of points $x =(x_1, \ldots, x_{n+1})$; $ e_1, e_2, \ldots, e_{n+1} $ are the coordinate unit vectors; $x\cdot y$ denotes the usual dot product. All points in $\bbr^{n+1}$ are identified with the corresponding column vectors. For the sake of simplicity, we use the same notation  ``$\,0\,$'' for the zero vector or a number and for the origin of the corresponding coordinate space. Precise meaning of this notation is  clear from the context.

We write $\bbb^{n+1}=\{x \in \bbr^{n+1}: |x|<1\}$ for  the open unit ball in $\bbr^{n+1}$; $\sn$ is the boundary of $\bbb^{n+1}$.  If $x\in  \sn$, then $dx$ stands for the  Riemannian  measure on $\sn$.  The notation $d\sig (x)$ is used for the lower-dimensional surface area measure on subspheres of $\sn$.
 Given a point $a\in \bbr^{n+1}\setminus \{0\}$, let $a^*=a/|a|^2$  be  the inversion of $a$ with respect to $\sn$, $\tilde a=a/|a| \in \sn$,   and $a^\perp$ be  the linear subspace orthogonal to the  vector $a$.

We set $\frM_{n,m}$ to be the space of all real matrices having $n$ rows and $m$  columns;
  $\mathrm{M}'$ denotes the transpose of the matrix $\mathrm{M}$; $\mathrm{I}_m$ is the identity $m \times m$   matrix.  For $n\geq m$,  $\stnm= \{\mathrm{M} \in \frM_{n,m}: \mathrm{M}'\mathrm{M}=\mathrm{I}_m \}$ is the Stiefel manifold of orthonormal $m$-frames in $\bbr^n$. Given $a \in \bbr^{n+1} \setminus \{0\}$, we write  ${\rm St} (a^\perp, m)$ for  the subset of  ${\rm St} (n+1, m)$ formed by $m$-frames in $a^\perp$.

 We denote by
$\T_a (n+1,k)$    a  family of all  $k$-dimensional affine planes $\t$ in $\bbr^{n+1}$ passing through the point  $a$ and intersecting the  ball  $\bbb^{n+1}$.  If $a\ne 0$, the notation $\Z_a (n+1,k)$ is used for the set of all  $k$-dimensional affine planes $\z$ that are parallel to  the  vector $a$ and meet  $\bbb^{n+1}$.  We also set
 \be\label {mvvcy} \mathrm {P}_a  x = \frac {a\cdot x}{|a|^2}\,a, \qquad   \mathrm {Q}_a=\mathrm {I}_{n+1} -\mathrm {P}_a,\qquad a\ne 0,\ee
for the corresponding projection maps.

The standard notation $C(\sn)$ and $L^p (\sn)$ is used for the  corresponding spaces of continuous and $L^p$ functions on  $\sn$.

\subsection {The Radon-John  Transform}\label {ppmjk}

Let $d$ and $n$ be  positive integers, $d< n$.
We denote by  ${\rm Gr}(n, d)$   the
 set of all $d$-dimensional affine planes  in
$\rn$; ${\rm Gr}_0 (n, d)$ is the  Grassmann manifold
 of  $d$-dimensional  linear subspaces  of $\rn$ .
 Each $d$-plane $\tau$  is parameterized as
$\t=\t(\t_0, u)$, where $\t_0 \in {\rm Gr}_0 (n, d)$ and $ u \in \t_{0}^\perp$, the orthogonal complement of $\t_0$ in $\rn$.
 The  Radon-John $d$-plane transform
 of a function $\vp$ on $\bbr^n$  is defined
 by   \be \label{vt3}(R\vp)(\tau) \equiv (R\vp)(\t_0, u) = \intl_{\displaystyle{\t_0}} \vp(u+v)\,
dv\ee
 provided  this integral is meaningful; see, e.g., \cite{GGG1, H11, Mar, Ru04b}.

 A great deal of  inversion formulas for  this  transform  is known. We recall  one of them, following \cite{Ru13b}.
  Consider the Erd\'elyi-Kober type fractional integral
\[(I^\a_{-, 2} f)(t)\! =\! \frac{2}{\Gamma (\alpha)} \intl_t^{\infty}\!
\frac{f(s) \, s\,ds} {(s^2-t^2)^{1- \alpha}}, \qquad \a>0. \]
  The relevant fractional derivative of order $d/2$ can be defined by
\[\Cal D^{d/2}_{-, 2} f = t^{2-d+2m}
(- D)^{m +1} t^{d} g, \quad g=I^{1-d/2+m}_{-,2} \,t^{-2m-2}\,f, \quad D=\displaystyle{\frac {1}{2t}\,\frac {d}{dt}},\]
where  $m=[d/2]$ is the integer part of $d/2$  and the powers of $t$  stand for the corresponding
 multiplication operators. If $d$ is even, this  formula can be replaced by the simple one:
\[
\Cal D^{d/2}_{-, 2} f=(- D)^{d/2} f.\]
Given  a function $\Phi$ on $ {\rm Gr}(n, d)$ and a point $x\in \rn$, we set
 \[(R^*_x \Phi)(t)= \intl_{SO(n)} \!  \Phi (\gam \t_0 +x + t\gam e_n) \,
 d\gam,  \qquad t>0,\]
 which is the mean value of $\Phi$ over all $d$-planes at distance $t$ from $x$. Here $\t_0$ is an arbirary $d$-plane passing through the origin.

 \begin{theorem}\label{invr1p}  {\rm (cf. \cite[Theorem 3.5]{Ru13b})}
A function $\vp \in L^p (\bbr^n)$,  $1\le p<n/d$, can be recovered from $\Phi=R\vp$   by the  formula
\be\label{nnxxzz}\vp(x) \equiv (R^{-1} \Phi)(x)=  \lim\limits_{t\to 0}\, \pi^{-d/2} (\Cal D^{d/2}_{-, 2} R^*_x \Phi)(t),\ee
where the  limit  is understood in the $L^p$-norm.
\end{theorem}

More inversion formulas in different classes of functions can be found in  \cite{H11,  Ru04b}.

\section {Parallel Slice Transforms}\label {ppmj}

We recall that given a point $a \in \bbr^{n+1}$, $|a|>1$, and an integer $k$, $1<k \le n$, the notation     $\Z_a (n+1,k)$ stands for the set of all  $k$-dimensional affine planes in $\bbr^{n+1}$ that are parallel to the vector $a$ and meet the open ball $\bbb^{n+1}$. Let
$\Z_a^0 (n+1,k)$ be the subset of  $\Z_a (n+1,k)$ that consists of all planes containing both $a$ and the origin.
 Every plane $\z \in \Z_a (n+1,k)$
 has the form
\be\label{amunl} \z\equiv \z(\z_0,u)=\z_0 +u,\ee
where \[ \z_0 \in \Z_a^0 (n+1,k), \qquad u \in \z_0^\perp \cap a^\perp, \qquad |u|<1.\]
Consider  the   parallel slice transform
  \be\label{munl}
(\Pi_{a} f)(\z)=\intl_{\sn \cap \,\z} f(x)\, d\sig (x),\qquad \z \in \Z_a (n+1,k),\ee
assuming $f\in C(\sn)$. This operator  annihilates all functions which are odd with respect to the subspace $a^\perp$.  Specifically,
let
 \be\label {lili1}
\mathrm {R}_a x = x-2\, \frac{x\cdot a}{|a|^2}\, a\ee
be  the reflection (or inversion) mapping in the hyperplane $a^\perp$. We set
\be\label {lili2}
C^{\pm}_a (\sn)= \left\{ f\in C(\sn): f(x)=\pm f (\mathrm {R}_a x)\right \},\ee
\be\label {lili2a}
f^+_a= \frac{f+ f\circ \mathrm {R}_a}{2}, \qquad f^-_a= \frac{f- f\circ \mathrm {R}_a}{2}.\ee
Then for  $ f\in C_a^-(\sn)$ we have $\Pi_{a} f =0$. In other words,   $\Pi_{a} f =0$ whenever $f^+_a=0$.
As we shall see below,
 the last equality characterizes the kernel of the operator $\Pi_{a}$. To prove this fact, we establish  connection between $\Pi_{a}$ and the  Radon-John transform $R$. This connection also yields a series of  inversion formulas for $\Pi_{a}$.

\begin{lemma} \label{lem2.1} Let $\z\equiv \z(\z_0, u) \in \Z_a (n+1,k)$. Then
\be\label {awlili2a}
(\Pi_{a} f)(\z)= (1-|u|^2)^{1/2} (R\vp)(a^\perp \cap \z),\ee
where $\vp$ is a function on the ball $\{y\in a^\perp: |y|<1\}$ defined by
\be\label {awlili2as} \vp(y)=2(1-|y|^2)^{-1/2} f(y+ \sqrt{1-|y|^2} \, \tilde a), \qquad \tilde a =\frac{a}{|a|},\ee
and $R$ is the Radon-John transform over  $(k-1)$-planes in  $a^\perp$.
\end{lemma}
\begin{proof}  To simplify the geometry of our consideration, let
\[\bbr^k=\bbr e_1 \oplus \cdots \oplus \bbr e_{k-1} \oplus \bbr e_{n+1}=\bbr^{k-1} \oplus \bbr e_{n+1}\] and choose $\gam \in O(n+1)$ so that
\[ \gam e_{n+1} =\tilde a, \quad \gam \bbr^k =\z_o, \quad \gam e_k=u/|u|.\]
We set
\[t=|u|, \qquad x=\gam (\xi +te_k), \qquad \xi \in \sn \cap \bbr^k.\]
The new variable $\xi$ ranges on the  sphere $\bbs^{k-1}_* $ of radius $\sqrt{1-t^2}$, centered at the origin, and lying
in the coordinate plane $\bbr^k$. Thus we have
\[
(\Pi_{a} f)(\z)= \intl_{\bbs^{k-1}_*} f(\gam (\xi +te_k)) \, d\sig (\xi).\]
This integral can be computed in a standard way, using projection $\eta$ of $\xi$ onto the  coordinate plane  $\bbr^{k-1}$. This gives
\bea
&&(\Pi_{a} f)(\z)= 2\intl_{|\eta|< \sqrt{1-t^2}}  f(\gam (\eta +te_k+ \sqrt {1-t^2 -|\eta|^2}\,e_{n+1})) \nonumber\\
&&\times\sqrt{1\!+\!\left (\frac{\partial \eta_{n+1}}{\partial \eta_1}\right )^2\!+\!\ldots \!+\!\left (\frac{\partial \eta_{n+1}}{\partial \eta_{k-1}}\right )^2} \,d\eta\nonumber\\
&&= 2\sqrt{1-t^2}\!\intl_{|\eta|< \sqrt{1-t^2}} \!\! f(\gam (\eta \!+\!te_k\!+ \!\sqrt {1\!-\!t^2 \! -\!|\eta|^2}\,e_{n+1}))\, \frac{d\eta}{\sqrt {1\!-\!t^2\! -\!|\eta|^2}}. \nonumber\eea
Denoting the argument of $f$ by $y$ and returning to the original notation, we obtain
\[
(\Pi_{a} f)(\z)\equiv (\Pi_{a} f)(\z_0,u) = 2\sqrt{1-|u|^2}\intl_{\t}
\frac{ f(y+ \sqrt{1-|y|^2} \, \tilde a)}{\sqrt{1-|y|^2}}\, dy,\]
where $\t =\{y\in  a^\perp \cap \z:  \, |y|<1,  {\rm P}_u y=u \}$,  ${\rm P}_u y$ being the orthogonal projection of $y$ onto the direction of $u$. Now the result follows.
\end{proof}

\begin{theorem} \label {jsbsa78}
The  operator $\Pi_{a}$ is injective on the subspace $C^+_a (\sn)$. Moreover,
\be\label {lili3}
\ker (\Pi_{a}) =C^{-}_a (\sn)=\{ f\in C(\sn): f^+_a=0\} .\ee
\end{theorem}
\begin{proof} The embedding $C^{-}_a (\sn) \subset \ker (\Pi_{a})$ is obvious. Conversely,
 if $\Pi_{a} f=0$ for some $f\in C(\sn)$, then $\Pi_{a} f^+_a = \Pi_{a} (f - f^-_a)=0-0=0$, and therefore, by Lemma \ref{lem2.1}, the corresponding function $\vp$ in (\ref{awlili2as}) is zero.  The latter means that $f^+_a  =0$.
 It remains to note that the injectivity of
$\Pi_{a}$  on  $C^+_a (\sn)$ follows from Theorem \ref{invr1p} for the Radon-John transform when the latter acts on  $(k-1)$-planes in  $a^\perp \sim \bbr^n$.
\end{proof}

The inversion formula for  $\Pi_{a} f$ on functions $f \in C^+_a (\sn)$ follows from the corresponding result for the Radon-John transform.   We define the operator $\Pi_{a}^{-1}: \Pi_{a}(C(\sn)) \to C_a^{+}(\sn)$ by setting
\begin{equation}\label{E:F-1nh}
 \Pi_{a}^{-1}g=f_a^+,  \qquad  g=\Pi_{a}f \in \Pi_{a}(C(\sn)).
\end{equation}
 This definition does not depend on the choice of the representative $f$ in  $g=\Pi_{a}f$. Indeed, if  $g=\Pi_{a}f_1=\Pi_{a}f_2$, then $f_1-f_2 \in \ker (\Pi_{a})$,  and therefore
\bea 0&=&\Pi_{a}f_1 -\Pi_{a}f_2 =\Pi_{a}[(f_1)_a^+ - (f_2)_a^+] + \Pi_{a}[(f_1)_a^- - (f_2)_a^-]\nonumber\\
&=&\Pi_{a}[(f_1)_a^+ - (f_2)_a^+].\nonumber\eea
Because the restriction  of $\Pi_{a}$ onto  $C_a^{+}(\sn)$ is injective, it follows that $ (f_1)_a^+ =(f_2)_a^+$.

If $g \in \Pi_{a}(C(\sn))$, that is, $g=\Pi_{a}f$ for some $f \in C(\sn)$, then
\[
\Pi_{a}(\Pi_{a}^{-1}g)=\Pi_{a}f_a ^+ = \Pi_{a}f=g,\] so that $\Pi_{a}^{-1}$ is the right inverse of the operator $\Pi_{a}$,  when the latter  acts on $C(\sn)$. However,  for functions $f\in C_a^{+}(\sn)$ we have $\Pi_{a}^{-1} \Pi_{a} f=f$, which means that $\Pi_{a}^{-1}$ can also be considered as the  left inverse of the restriction of $\Pi_{a}$ onto $C_a^{+}(\sn)$. In the last setting, an explicit expression for $\Pi_{a}^{-1}$ can be obtained by making use of the following theorem.

\begin {theorem} \label {lkuyr} Let $1 \!<k \le n$, $a \!\in \!\bbb^{n+1}\!\setminus \{0\}$.
A function $f \in C^+_a (\sn)$  can be reconstructed from $g(\z)=(\Pi_{a} f)(\z)$ by the formula
\be\label{xcmnvb}
f(x) \equiv (\Pi_{a}^{-1} g)(x)=\frac{1}{2} \,|x\cdot \tilde a| \,(R^{-1} \Phi) ({\rm Q}_{a} x), \qquad \tilde a =\frac{a}{|a|},\ee
where  ${\rm Q}_{a} x$ is  the orthogonal projection of $x$ onto $a^\perp$,
\be\label{xcmnvb1}
\Phi (\t)\equiv \Phi (\t_0, u)= (1-|u|^2)^{-1/2} g ((\t_0 \oplus \bbr \tilde a) +u),   \ee
\[ u\in \t_0^\perp\cap \,\bbb^{n+1},\]
 $R^{-1} \Phi$ is the inverse Radon-John transform of $\Phi$ on $a^\perp$ (cf.   (\ref{nnxxzz})).
\end{theorem}
\begin {proof}  Let $\sn_{a+}= \{x \in \sn: x \cdot a\ge 0\}$.
Because $f \in C^+_a (\sn)$, it suffices to prove (\ref{xcmnvb}) for $x\in \sn_{a+}$.
By Lemma \ref{lem2.1},  the function $\Phi$  in (\ref{xcmnvb1}) is represented by the Radon-John transform $R\vp$ with
\be\label{xcmnvb2}
\vp(y)=2(1-|y|^2)^{-1/2} f(y+ \sqrt{1-|y|^2} \, \tilde a).\ee
The latter belongs to $L^1$ on the ball $\{y\in a^\perp: |y|<1\}$ and is continuous inside this ball. Now, because $(1-|y|^2)^{1/2}=x\cdot \tilde a$, the result follows.
\end{proof}

\section {Basic  Automorphisms of $\sn$}

 Given a point $a\in \bbr^{n+1}\setminus \{0\}$, we define the  projection maps
 \[\mathrm {P}_a  x = \frac {a\cdot x}{|a|^2}\,a, \qquad   \mathrm {Q}_a=\mathrm {I}_{n+1} -\mathrm {P}_a.\]
Let first $|a|<1$. Consider the map
 \be\label{E:g_a}
\vp_a x=\frac{a-\mathrm {P}_a x - s_a \mathrm {Q}_a x }{1-x\cdot a}, \qquad s_a=\sqrt{1-|a|^2},\ee
which  is a one-to-one M\"obius transformation satisfying
\be\label{Evf} \vp_a (0)=a, \qquad \vp_a (a)=0, \qquad  \vp_a (\vp_a x)=x,\ee
\be\label{Evf1}
1-|\vp_a x|^2=\frac{(1-|a|^2)(1-|x|^2)}{(1-x\cdot a)^2}, \qquad x\cdot a \neq 1;\ee
see, e.g., Rudin
\cite[Section 2.2.1]{Rud}, Stoll \cite[Section 2.1]{St}.
By (\ref{Evf1}), $\vp_a$  maps the ball $ \bbb^{n+1}$ onto itself and  preserves $\sn$.

In the case $|a|>1$ we replace $a$  by its Kelvin transformation $a^* =a/|a|^2$, so that  $|a^*|<1$. Then
\be\label{Evfp} \vp_{a^*} (0)=a^*, \qquad \vp_{a^*} (a^*)=0, \qquad  \vp_{a^*} (\vp_{a^*} x)=x\ee
and   $\vp_{a^*} (a)=\infty$ in the suitable  sense. The latter means that $\vp_{a^*}$ maps planes through $a$ to planes  parallel to the vector  $a$, and vice versa.

To give this statement  precise meaning,
let $\T_a (n+1,k)$   (or $\Z_a (n+1,k)$) be  the set of all  $k$-planes in $\bbr^{n+1}$ intersecting the open unit ball $\bbb^{n+1}$ and passing through  $a$  (or parallel to the vector $a$, resp.). It is convenient  to use the language of Stiefel manifolds. Specifically, every
 plane  $\t\in \T_a (n+1,k)$ has the form
\be\label{kuytr} \t\equiv \t_\xi=\{ x\in \bbr^{n+1} :\, \xi'x =\xi'a\}, \ee
\[\xi\in  \stnk, \qquad |\xi'a|<1.\]
Similarly, every plane  $\z\in \Z_a (n+1,k)$ is given by
\be\label{kuy44tr} \z\equiv \z_{\eta,t}=\{ y\in \bbr^{n+1} :\, \eta'y =t\}, \ee
\[ \eta\in  {\rm St} (a^\perp, n+1-k), \qquad  t\in \bbr^{n+1-k}, \qquad  |t|<1.\]

\begin{lemma} \label{lemuupz}  Let   $|a|>1$,  $1\le k \le n$.
The involutive map $\vp_{a^*}$ extends  as a bijection from  $\T_a (n+1,k)$ onto   $\Z_a (n+1,k)$. Specifically, if $\t$ has the form (\ref{kuytr}), then $\z=\vp_{a^*} \t$ has the form (\ref{kuy44tr})
with
\be\label {obbmka} \eta = (\mathrm {Q}_{a^*}\xi) \rho^{-1/2}, \qquad \rho= (\mathrm {Q}_{a^*} \xi)'(\mathrm {Q}_{a^*}\xi), \ee
\be\label {oujpa}
 t= - s_{a^*}  \rho^{-1/2}\xi'a.\ee
Conversely, if $\z$ has the form (\ref{kuy44tr}), then $\t=\vp_{a^*} \z$ has the form (\ref{kuytr}) with  \be\label {obbmkb} \xi = (s_{a^*} \eta + a^*t') s^{-1/2}, \qquad s=s_{a^*}^2
\mathrm {I}_{n+1-k} +|a^*|^2 tt'.
 \ee
\end{lemma}
\begin{proof} Let $\t\in \T_a (n+1,k)$ be defined by  (\ref{kuytr}).
Because $|\xi'a|$ is the Euclidean distance from $\t$ to the origin, the condition $|\xi'a|<1$ is equivalent to $\t \cap \bbb^{n+1} \neq \emptyset$. The $\vp_{a^*}$-image of $\t$ is
\[
\vp_{a^*} \t=\{ y \in \bbr^{n+1} :\, \xi'[\vp_{a^*} y -a]=0 \}.\]
By (\ref{E:g_a}),
\be\label {irrr} \vp_{a^*} y -a=-\frac{s_{a^*}}{1-y\cdot  a^*}\, (\mathrm {Q}_{a^*} y + s_{a^*} a),\ee
and therefore
\be\label {oi7t}
\vp_{a^*} \t=\{ y \in \bbr^{n+1} :\, (\mathrm {Q}_{a^*} \xi)'y = -s_{a^*} \xi'a \}.\ee
We represent the $(n+1)\times (n+1-k)$ matrix  $\mathrm {Q}_{a^*} \xi$  in the polar form \be\label{lzxg}
 \mathrm {Q}_{a^*} \xi=\eta \rho^{1/2};\ee
 see, e.g., \cite[pp. 66, 591]{Mu}. Here $\rho= (\mathrm {Q}_{a^*} \xi)'(\mathrm {Q}_{a^*}\xi)$ is a positive definite $(n+1-k)\times (n+1-k)$ matrix and
 \be\label{luug}
\eta = (\mathrm {Q}_{a^*}\xi) \rho^{-1/2}\in {\rm St} (a^\perp, n+1-k). \ee
 Now  (\ref{oi7t}) can be written as
\be\label {oi7tmk}
\vp_{a^*} \t=\{ y \in \bbr^{n+1} :\, \eta'y = t\}, \qquad t= - s_{a^*}  \rho^{-1/2}\xi'a.\ee
Because $\eta$ is orthogonal to the vector $a$, the $k$-plane (\ref{oi7t}) is parallel to $a$. Further, we observe that
\be\label {oi7ta }
|t|= s_{a^*} |\rho^{-1/2}\xi'a|<1\ee
because the $\vp_{a^*}$-image of any line interval in $\t \cap \bbb^{n+1}$ is the line interval in $\vp_{a^*} \t \cap \bbb^{n+1}$.
It follows that  the $k$-plane $\vp_{a^*} \t $ has non-empty intersection with $\bbb^{n+1}$, and therefore,  $\vp_{a^*} \t \in \Z_a (n+1,k)$.

Conversely, let  $\z\in \Z_a (n+1,k)$ be defined by  (\ref{kuy44tr}).
Then $\vp_{a^*} \z=\{x \in \bbr^{n+1} :\, \eta' \vp_{a^*} x =t \}$.
Using (\ref{E:g_a}) and keeping in mind that $\eta$ is orthogonal to the vector $a$, we obtain
\[
\eta' \vp_{a^*} x =-\frac{s_{a^*} \eta' x}{1-x \cdot a^*}.\]
Hence the equality $\eta' \vp_{a^*} x =t$ is equivalent to
$s_{a^*} \eta' x =t (1-x \cdot a^*)$
or
\be\label{mzwwq} (s_{a^*} \eta'  +t(a^*)')x=t.\ee
Setting $\tilde \xi=s_{a^*} \eta + a^*t'$, we can write (\ref{mzwwq}) in the required form as  $\tilde \xi'x=\tilde \xi' a$.
To complete the proof, it remains to normalize $\tilde \xi$ by writing it  in the polar form
\[\tilde \xi =\xi s^{1/2}, \quad \xi \in \stnk, \quad s=\tilde \xi'\tilde \xi=s_{a^*}^2
\mathrm {I}_{n+1-k} +|a^*|^2 tt'.\]
\end{proof}

Now we   introduce another  automorphism that plays an important role in our paper.  Given two distinct points $a$ and $x$, let $L_{a,x}$ be the straight line  through these points.  Assuming  $|a|>1$  and $x \in \sn$, we  define the  reflection map
\[\tau_a :\sn \to \sn, \quad x  \to \tilde x,\]
 where $\tilde x$ is the second point at which  $L_{a,x}$ meets  $\sn$. If $L_{a,x}$ is  tangent  to $\sn$, we simply set $\tilde x =x$. A straightforward computation gives
  \be\label {jhbcr1}
\tilde x=\tau_a x =\frac{(|a|^2 -1)\, x + 2(1-x\cdot a)\, a }{|x -a|^2}.\ee

\begin{lemma}\label{L:muaR} Let  $\mathrm {R}_a= \mathrm {I}_{n+1}-2 \mathrm {P}_a$ be  the reflection  in  $a^\perp$, as in (\ref{lili1}). Then
\be \label{cn4hos}  \tau_a x=\vp_{a^*}\mathrm {R}_a \vp_{a^*} x, \qquad  x \in \sn.\ee
\end{lemma}
\begin{proof}
By Lemma  \ref{lemuupz},
 the map $\vp_{a^*}$ transforms  the segment $[x, \tau_ax]$ into the  segment
 $[\vp_{a^*} x, \vp_{a^*} \tau_a x]$ on the line $\vp_{a^*} L_{a,x}$, which is parallel to the vector  $a$. Hence, by  the definition of the reflection $\mathrm {R}_a$, we have $\mathrm {R}_a \vp_{a^*} x=\vp_{a^*}\tau_a x.$ Now (\ref{cn4hos}) follows because $\vp_{a^*}$ is an involution.
\end{proof}

\section{Injectivity and Inversion of $F_a$}

Our treatment of the shifted Funk transform
 \[(F_a f)(\t) = \intl_{\sn \cap \,\t} f(x)\, d\sig (x), \quad \t \in \T_a (n+1,k), \quad |a|>1,\]
relies on the following theorem establishing connection between $F_a$ and
 the parallel slice transform $\Pi_a$.  For $y\in \sn$ we denote
\be\label{vt1de}
(M_{a^*}f)(y)\!=\!\left ( \frac {s_{a^*} }{1\!-\!a^*\cdot  y}\right )^{k-1}\!\!(f\circ \vp_{a^*})(y), \quad s_{a^*}\!=\!\sqrt{1\!-\!|a^*|^2}, \ee
 where $\vp_{a^*}$ is the spherical automorphism, which is  defined by (\ref{E:g_a}) with $a$  replaced by $a^*=a/|a|^2$.

\begin{theorem}\label {ukdndt}
If  $f\in C(\sn)$,   then
\be\label{vt1ata}
(F_af)(\t)= (\Pi_a M_{a^*} f)(\vp_{a^*}\t), \qquad \t \in \T_a (n+1,k). \ee
\end{theorem}

We prove this theorem in the next section and now continue our investigation, taking into account that $M_{a^*}$ is one-to-one and its inverse has the form
 \be\label{tyhce1}
 (M_{a^*}^{-1} f)(x)=s_{a^*}^{1-k}(1-a^* \cdot \vp_{a^*} x )^{k-1}\, (f \circ \vp_{a^*})(x).\ee

 Injectivity of $\Pi_a$  was discussed in  Theorem \ref{jsbsa78}.
It follows that $F_a$
is injective on functions $f\in C(\sn)$ if and only if
\be\label{mmffju}
(M_{a^*} f)(y) = (M_{a^*} f)(\mathrm {R}_a y), \ee
or
\be\label{mmju} (f\circ \vp_{a^*})(y)=\left (\frac{1\!-\!a^*\cdot  y}{1-  a^* \cdot \mathrm {R}_a y}\right )^{k-1} (f\circ\vp_{a^*})(\mathrm {R}_a y).\ee
Further, the kernel of $F_a$ is formed by functions $f$, for which $M_{a^*} f$ belongs to the kernel of $\Pi_a$, that is $(M_{a^*} f)(y) = - (M_{a^*} f)(\mathrm {R}_a y)$.
Setting $y=\vp_{a^*} x$,
\be\label {nxgbsa1}
\rho_{a^*} (x)=\left (\frac{1\!-\!a^* \cdot \vp_{a^*} x}{1\!-\!a^* \cdot \mathrm {R}_a \vp_{a^*} x}\right )^{k-1},  \ee
and noting that $\vp_{a^*} \mathrm {R}_a \vp_{a^*} x=\t_a x$ (see Lemma \ref{L:muaR}),
we  can write (\ref{mmffju}) in the equivalent form
\be\label{iUty0}
f(x)= (W_a f)(x), \qquad (W_a f)(x)=\rho_{a^*} (x) f(\t_a x). \ee
The weight function $\rho_{a^*} (x)$ can be represented in a simpler form
\be\label {nxgbsa1p}
\rho_{a^*} (x)=\left (\frac{1-|a^*|^2}{|a^*- x|^2}\right )^{k-1}  =\left (\frac{|a|^2 -1}{|a- x|^2}\right )^{k-1} \ee
(we leave straightforward computation to the reader).

\begin{lemma} \label {u54aq} The operator $W_{a}$ is an involution, i.e., $W_{a}W_{a} f=f$.
\end{lemma}
\begin{proof} Taking into account that $\tau_a\tau_a x=x$,
 we have
$(W_{a}W_{a} f)(x)=\rho_{a^*} (x) \rho_{a^*} (\tau_a x )\, f(x)$.  By (\ref{nxgbsa1}), $\rho_{a^*} (x) \rho_{a^*} (\tau_a x )=1$. This gives the result.
\end{proof}

The above reasoning shows that $W_{a}$ serves as a reflection map in our consideration and  the following statements hold.

\begin{theorem} \label{luqw} Let $1<k\le n$.
\begin {enumerate}

\item [(a)]
The operator $F_a$ is injective on functions $f \in C(\sn)$ satisfying the functional equation
\be\label{ixy0}  f=  W_a f.\ee

\item [(b)] The kernel of $F_a$ has the form
\be\label{ixy0a} \ker (F_a)= \{ f \in C(\sn) : \,f= - W_a f\}.\ee

\end{enumerate}
\end{theorem}

\begin{theorem}\label{i00s} Let  $a\in \bbr^{n+1}$, $|a|>1$.
A   function $f\in C(\sn)$ satisfying  (\ref{ixy0}) can be uniquely reconstructed from $g=F_a f$ by the formula
\be\label{irxy0} f=M_{a^*}^{-1}\Pi_a^{-1} g_a, \qquad g_a= g\circ \vp_{a^*},\ee
 where  $M_{a^*}^{-1}$ and $\Pi_a^{-1}$ are defined by  (\ref{tyhce1}) and (\ref{xcmnvb}),  respectively.
\end {theorem}

 \section{Proof of Theorem  \ref{ukdndt}}

  We proceed as in our previous work \cite{AR}, however, with several  changes caused by the essence of the matter.
Let  $\t  \in \T_a (n+1,k)$ have the form
(\ref{kuytr})  and write
 \[(F_a f)(\t_\xi)=\intl_{\{x\in \sn: \,\xi' (x-a)=0\}}\!\! \!\!\!f(x)\, d\sig (x), \qquad \xi\in  \stnk.\]
 By Lemma \ref{lemuupz},  it suffices to show that
 \be\label{myywq}
(F_a f)(\t_\xi)= (\Pi_a M_{a^*} f)(\z_{\eta,t}),\ee
 where $\t_\xi$ and $\z_{\eta,t}=\vp_{a^*}\t_\xi$ are related by (\ref{obbmka})-(\ref{oujpa}).
We set
\be\label{kuyb}
 (F_{a, \e} f)(\xi)=\intl_{\sn} f(x)\, \om_\e (\xi'(x-a))\, dx,\qquad \e >0,\ee
where $\om_\e$ is a radial bump function supported on the ball in $\bbr^{n+1-k}$ of radius $\e$ with center at the origin, so that for any function $g$, which is continuous in a neighborhood of the origin, we have
\[\lim\limits_{\e\to 0} \intl_{|t|<\e}\om_\e (t)\,g(t)\,dt=g(0).\]
One can  show that
\be\label{kuyb1}
 \lim\limits_{\e\to 0} (F_{a, \e} f)(\xi)=(1-|\xi' a|^2)^{-1/2}(F_{a} f)(\t_\xi).\ee
The proof of this equality is a verbatim copy of Step I in the similar proof in \cite [Section 6]{AR}, and we skip it.

 Let us obtain an alternative expression for the limit
(\ref{kuyb1}), now in terms of the automorphism $\vp_{a^*}$. We make use of the equality

\be\label{fff}\intl_{\sn} f(x)\, dx= s_{a^*}^n \intl_{\sn} \frac{(f\circ \vp_{a^*})(y)}
{(1-  a^* \cdot y)^n}\, dy,\ee
which was proved in \cite[Lemma 2.1]{AR}.
By this equality,
\[
 (F_{a, \e} f)(\xi)= s_{a^*}^n \intl_{\sn} \frac{(f\circ \vp_{a^*})(y)}
{(1-  a^* \cdot y)^n}\, \,\om_\e (\xi'[\vp_{a^*} y-a])\, dy,\]
where (cf. (\ref{irrr}))
\[
\xi'[\vp_{a^*} y-a]= -\frac{s_{a^*}}{1-  a^* \cdot y}\,  [(\mathrm {Q}_{a^*} \xi)'y + s_{a^*} \xi'a].
\]
Denote
\[ \tilde f (y)= s_{a^*}^n \frac{(f\circ \vp_{a^*})(y)}
{(1-  a^* \cdot y)^n}, \qquad t_0= - s_{a^*} \xi'a,\]
and write $\mathrm {Q}_{a^*} \xi$ in the polar form
\be\label {jnlli} \mathrm {Q}_{a^*} \xi=\eta \rho^{1/2}, \quad \rho= (\mathrm {Q}_{a^*} \xi)'(\mathrm {Q}_{a^*} \xi), \quad \eta = (\mathrm {Q}_{a^*} \xi)  \rho^{-1/2}.\ee
This gives
\[
 (F_{a, \e} f)(\xi)= \intl_{\sn} \tilde f (y)\, \om_\e
 \left(\frac{s_{a^*}}{1-  a^* \cdot y}\, (\rho^{1/2} \eta'y -t_0)\right ) dy.\]
The frame $\eta \in \stnk$ spans an $(n+1-k)$-dimensional subspace that will be denoted by $\{\eta \}$.

Let us pass to bispherical coordinates    (see, e.g., \cite [p. 31]{Ru15})
\[ y= \left[\begin{array} {c}
\vp \, \sin \th \\ \psi\, \cos \,\th \end{array} \right], \quad   \vp \in \sn \cap \eta^\perp, \quad \psi \in \sn \cap \{\eta\}, \quad 0 \!\le\!\theta \!\le\! \pi/2,,\]
\[
 \qquad  dy=
 \sin^{k-1}\theta \, \cos^{n-k}\theta \, d\theta d\vp d\psi,\]
and set $s=\cos\, \th$. We obtain
\bea
&&(F_{a, \e} f)(\xi)= \intl_{0}^1 s^{n-k} (1-s^2)^{(k-2)/2} ds \intl_{\sn \cap \eta^\perp} d\vp \nonumber\\
&&\times \intl_{\sn \cap\{\eta\}}
\tilde f \left(\left[\begin{array} {c}
\sqrt{1-s^2} \vp   \\ s\psi\ \end{array} \right] \right )\, \om_\e
\left(\frac{s_{a^*} ( \rho^{1/2} s\psi -t_0)}{1-a^* \cdot  (\sqrt{1-s^2}\, \vp+s\psi)}\right ) \, d\psi,\nonumber \eea
or (set $z=s\psi \in \{\eta\}$, $|z|<1$)
\bea
&&(F_{a, \e} f)(\xi)= \intl_{\bbb^{n+1} \cap \{\eta\}} (1-|z|^2)^{(k-2)/2} dz  \nonumber\\
&&\times
\intl_{\sn \cap \eta^\perp} \tilde f \left(\left[\begin{array} {c}
 \sqrt{1-|z|^2} \vp  \\ z\ \end{array} \right] \right )\, \om_\e
\left(\frac{s_{a^*} ( \rho^{1/2} z -t_0)}{1-a^* \cdot  (\sqrt{1-|z|^2}\, \vp+z)}\right )\, d\vp.\nonumber \eea
We set
\[
 \tilde f_\vp (z)=\tilde f \left(\left[\begin{array} {c}
 \sqrt{1-|z|^2} \vp  \\ z\ \end{array} \right] \right ), \quad \lam=s_{a^*} \rho^{1/2}, \quad
 \del= s_{a^*} t_0,\]
 \[ h_\vp (z)=1-a^* \cdot  (\sqrt{1-|z|^2}\, \vp+z),\]
to get
\[
F_{a, \e} f)(\xi)= \intl_{\bbb^{n+1} \cap \{\eta\}} (1-|z|^2)^{(k-2)/2} dz\intl_{\sn \cap \eta^\perp}\tilde f_\vp (z)\,
\om_\e
\left(\frac{\lam z -\del}{ h_\vp (z)}\right )\, d\vp.\]
Denote  $t=\lam ^{-1}  \del$ and change variable
\be\label {kana4}  r\equiv r_\vp(z)=\frac{\lam z -\del}{ h_\vp (z)},\ee
so that $r=0$ if and only if $z=t$. One can
 write (\ref{kana4}) as
\[
\Phi (r,z)\equiv\lam z -\del-r h_\vp (z) =0.\]
 Because the matrix
 $(\partial \Phi/\partial z)(0,0)=\lam$    is invertible, there exists an inverse function  $z=z_\vp(r)$, which satisfies $z_\vp(0)=t$ and is differentiable in a  small neighborhood of $r=0$.
  Hence, for sufficiently small $\e>0$,
\[
(F_{a, \e} f)(\xi)\!=\! \intl_{|r|<\e} \!\!  \om_\e (r) \,dr\!
\!\intl_{\sn \cap \eta^\perp}\!\!\!\tilde f_\vp (z_\vp(r))\,(1\!-\!|z_\vp(r)|^2)^{(k-2)/2} \, |\det (r'_\vp(r))| d\vp.\]
 Passing to the limit and noting that
 \[ |\det (z'_\vp(0))|=\frac{[h_\vp (t)]^{n+1-k}}{\det (\lam)}=
 \frac{[h_\vp (t)]^{n+1-k}}{s_{a^*}^{n+1-k}\, [\det (\rho)]^{1/2}},\]
  we obtain
\bea
\lim\limits_{\e\to 0} (F_{a, \e} f)(\xi)&=& \frac{(1\!-\!|t|^2)^{(k-2)/2}}{s_{a^*}^{n+1-k} \,[\det (\rho)]^{1/2}} \intl_{\sn \cap \eta^\perp}\! \!\tilde f\left(\left[\begin{array} {c}
 \sqrt{1\!-\!|t|^2} \vp  \\ t\ \end{array} \right] \right )\nonumber\\
&\times& \,
  [1-a^* \cdot  (\sqrt{1-|t|^2}\, \vp+t)]^{n+1-k} d\vp,\nonumber\eea
or
\bea
\lim\limits_{\e\to 0} (F_{a, \e} f)(\xi)&=& \frac{(1\!-\!|t|^2)^{(k-2)/2}}{[\det (\rho)]^{1/2}}
\intl_{\sn \cap \eta^\perp}\! \! (f\circ \vp_{a^*})(\sqrt{1\!-\!|t|^2} \vp + t) \nonumber\\
 &\times& \left (\frac{s_{a^*}}{1-a^* \cdot  (\sqrt{1-|t|^2}\, \vp+t)} \right )^{k-1} d\vp, \nonumber\eea
 Using the notation in (\ref{vt1de}) and scaling the measure $d\vp$, we can write the last formula  as \[ \lim\limits_{\e\to 0} (F_{a, \e} f)(\xi)= \frac{1}{h^{1/2}} \intl_{\eta' y=t} (M_{a^*}f)(y)\, d\sig (y),\]
where $h= (1-|t|^2)\,\det (\rho)$,   $\;t =-s_{a^*} \rho^{-1/2}\xi'a$, $\;\rho= (\mathrm {Q}_{a^*} \xi)'(\mathrm {Q}_{a^*} \xi)$.

Let us simplify the expression for $h$.
 We have
\[
1-|t|^2=1- s_{a^*}^2 \, (a'\xi \rho^{-1/2} )(\rho^{-1/2}\xi'a).\]
This expression can be transformed by making use  of
  the known fact  from Algebra (see, e.g., \cite [Theorem A3.5]{Mu}). Specifically,
  if $A$ and $B$ are  $p\times q$ and $q \times p$ matrices, respectively, then
\be\label{aooi4s}
\det (\mathrm {I}_{p} +AB) =\det (\mathrm {I}_{q} +BA).\ee
By this formula,
 \[1-|t|^2=\det  \left (\mathrm {I}_{n+1-k}-s_{a^*}^2 \, (\rho^{-1/2}\xi'a)(a'\xi \rho^{-1/2})\right ),\]
and therefore
 \[1-|t|^2=\det \left (\mathrm {I}_{n+1-k}-|a|^2 s_{a^*}^2 \, \rho^{-1/2}\xi'\mathrm {P}_a \xi  \rho^{-1/2}\right ),\]
where $\mathrm {P}_a$ stands for the orthogonal projection map onto the direction of the vector $a$.
Thus,
\[
h =\det \left (\rho-|a|^2 s_{a^*}^2 \, \xi'\mathrm {P}_a \xi\right ).\]
Because $\rho= (\mathrm {Q}_{a^*} \xi)'(\mathrm {Q}_{a^*} \xi)=\xi'\mathrm {Q}_{a} \xi$, we continue:
\bea
h &=& \det \left(\xi'\mathrm {Q}_a \xi - (|a|^2-1)\, \xi'\mathrm {P}_a \xi\right )=
\det \left(\mathrm {I}_{n+1-k}-|a|^2 \xi'\mathrm {P}_a \xi\right )\nonumber\\
 &=&\det \left(\mathrm {I}_{n+1-k}-\xi' a  a' \xi \right),\nonumber\eea
and (\ref{aooi4s}) gives
\[ h =1-(a' \xi)(\xi' a)=1-|\xi'a|^2.\]
It follows that
\[ \lim\limits_{\e\to 0} (F_{a, \e} f)(\xi)= (1-|\xi'a|^2)^{-1/2} \intl_{\eta' y=t} (M_{a^*}f)(y)\, d\sig (y),\]
Comparing this formula with  (\ref{kuyb1}), we obtain
\[(F_{a} f)(\t_\xi)= \intl_{\eta' y=t} (M_{a^*}f)(y)\, d\sig (y),\]
which coincides with (\ref{myywq}).

\section{ The Link Between Shifted Funk Transforms over Spherical Sections of Different Dimensions}

The link between integral-geometrical objects of different dimensions is a part of the  theory of these objects  that can be useful in many occurrences. It is intuitively clear that to integrate a function $f$ over an $\ell$-dimensional section, say $\a$, it suffices first to  integrate this function over each  section $\b$ of lower dimension $k<\ell$ and then inegrate the result over all $\b$ in $\a$. Specifically,
\be\label{fafaaa}  \intl_{x\in \a} f (x)\, d_a x= \intl_{\b \subset \a}  d\b \intl_{x\in \b} f (x)\, d_\b x, \ee
where integration is performed with respect to the corresponding measures.

This procedure is well known for diverse Radon-like transforms, including $F_a$ with $a=0$. Below we establish such connection for  transforms $F_af$  with arbitrary center $a$ in $\bbr^{n+1}$.

We slightly change our notation and write  $F_af$ as
 \be\label{fdl} (F_k f)(\xi; a)=\intl_{\{x \in \sn: \,\xi'(x -a)=0\}} \!\!\!f(x)\, d\sig (x), \qquad \xi \in \stnk.\ee
 As before, it is assumed that $|\xi'a|<1$, because otherwise, the integral  is identically zero.

Some more notations are in order.
In the sequel $\bbr^{n+1}=\bbr^{k} \times \bbr^{n-k+1}$,
\[ \bbr^{k} =\bbr e_1 \oplus \ldots \oplus \bbr e_{k},\qquad \bbr^{n-k+1} = \bbr e_{k+1 }\oplus \ldots \oplus \bbr e_{n+1},\]
  $\bbs^{k-1}$ is the unit sphere in  $\bbr^{k}$,
\[
\xi_0= \left[\begin{array} {c}
0 \\ I_{n+1-k} \end{array} \right]\in \stnk.\]
 Given $\xi \in \stnk$, we denote by $r_\xi$ an
 arbitrary rotation satisfying $ r_\xi  \xi_0=\xi$.
Then
 \be\label{fdl2}
 (F_k f)(\xi; a)=(1\!-\!|\xi'a|^2)^{(k-1)/2} \intl_{\bbs^{k-1}} f \left(r_\xi\left[\begin{array} {c}
\vp \, \sqrt{1\!-\!|\xi'a|^2} \\ \xi'a\end{array} \right] \right )d\vp. \ee
The corresponding normalized transform has the form
 \be\label{fdl3}
(\dot{F}_k f)(\xi; a)=\intl_{\bbs^{k-1}} f \left(r_\xi\left[\begin{array} {c}
\vp \, \sqrt{1\!-\!|\xi'a|^2} \\ \xi'a\end{array} \right] \right )d_*\vp,\ee
where $d_*\vp$ stands for the probability measure on $\bbs^{k-1}$.

Fix any integer  $\ell$, so that $k<\ell \le n$. Every frame  $\eta \in {\rm St} (n+1, n+1-\ell)$ extends to a ``bigger'' frame $\xi \in \stnk$, so that $\xi= [\tilde \eta, \eta]$, $\tilde \eta \in {\rm St} (\eta^\perp, \ell -k)$, where ${\rm St} (\eta^\perp, \ell -k)$ is the Stiefel manifold of  all $(\ell -k)$-frames in the $\ell$-dimensional subspace $\eta^\perp$. Here the notation $[\tilde \eta, \eta]$ is used for the
$(n+1)\times (n+1-k)$ matrix composed by the columns of $\tilde \eta$ and $\eta$.

\begin{theorem} \label{mgvdr} Let $f\in L^1(\sn)$, $a\in\bbr^{n+1}$,  $1<k<\ell \le n$. Then
 \be\label{fdl4}
(\dot{F}_\ell f)(\eta; a)= \intl_{ {\rm St} (\eta^\perp, \ell -k)} (\dot{F}_k f)([\tilde \eta, \eta]; a) \, d\tilde \eta.\ee
\end{theorem}
\begin{proof} We denote the left-hand side of (\ref{fdl4}) by $(Af)(\eta, a)$ and the right-hand side by
 $(Bf)(\eta, a)$. Setting $f_\gam =f\circ \gam$, $\gam \in O(n+1)$, we have
 \[(Af)(\gam \eta, \gam a)=  (Af_\gam)(\eta, a), \qquad
 (Bf)(\gam \eta, \gam a)=  (Bf_\gam)(\eta, a). \]
 Hence it suffices to prove (\ref{fdl4}) for the case when $\eta$
 is the coordinate frame of the subspace
\[ \bbr^{n-\ell+1} = \bbr e_{\ell+1 }\oplus \ldots \oplus \bbr e_{n+1},\]
that is, $\eta=\eta_0$, where
  \be\label{fdl5}
 \eta_0= \left[\begin{array} {c}
0\\ \mathrm{I}_{n+1-\ell} \end{array} \right]      \in {\rm St} (n+1, n+1-\ell).\ee
 In this case we can write $A\equiv (Af)(\eta_0, a)$ as
\be\label{fdl6}
A=\intl_{O(\ell)} f \left(\tilde \a
\left[\begin{array} {c}
e_1 \, \sqrt{1\!-\!|\eta_0'a|^2} \\ \eta_0'a\end{array} \right] \right )d\a, \quad \tilde \a =\left[\begin{array} {cc}
\a & 0  \\ 0 &  \mathrm{I}_{n+1-\ell} \end{array} \right].\ee
Here $e_1$ is interpreted as the first coordinate unit vector in
\[
\bbr^{\ell} =\bbr e_1 \oplus \ldots \oplus \bbr e_{\ell}.\]

Similarly, for $B\equiv (Bf)(\eta_0, a)$ we have
\[
B=\intl_{ {\rm St} (\ell, \ell -k)} (\dot{F}_k f)([\tilde \eta, \eta_0]; a) \, d\tilde \eta=\intl_{O(\ell)} (\dot{F}_k f)([\tilde \a\zeta_0, \eta_0]; a) \, d\a,\]
where  ${\rm St} (\ell, \ell -k)$ is the Stiefel manifold of orthonormal $(\ell-k)$-frames in $\bbr^{\ell}$,
\be\label{fdl7}
 \zeta_0= \left[\begin{array} {c}
0\\ \mathrm{I}_{\ell -k} \\ 0\end{array} \right]      \in {\rm St} (n+1, \ell -k).\ee

Let us transform $ (\dot{F}_k f)([\tilde \a\zeta_0, \eta_0]; a)$ using (\ref{fdl3}) with $\xi=[\tilde\a\zeta_0, \eta_0]$. In this case,
$r_\xi$ takes $\xi_0$ to $[\tilde\a\zeta_0, \eta_0]=\tilde\a [\zeta_0, \eta_0]= \tilde\a \xi_0$. Hence $\xi'a=(\tilde \a \xi_0)' a$, and we can write
\[
 (\dot{F}_k f)([\a\zeta_0, \eta_0]; a)=\intl_{\bbs^{k-1}} f(\tilde \a v_\psi )\,d_*\psi,  \quad
v_\psi=\left[\begin{array} {c}
\psi \, \sqrt{1\!-\!|(\tilde \a  \xi_0)' a|^2} \\(\tilde \a \xi_0)' a \end{array} \right].
 \]
Setting
\[\psi= \om e_1, \qquad \om \in O(k), \qquad \tilde \om= \left[\begin{array} {cc}
\om & 0  \\ 0 &  \mathrm{I}_{n+1-k}\end{array} \right],\]
and noting that $\xi_0= \tilde \om\xi_0$, we obtain
\[
v_\psi=\tilde \om v_0, \qquad v_0=
\left[\begin{array} {c}
e_1 \, \sqrt{1\!-\!|(\tilde \a \tilde \om \xi_0)' a|^2} \\(\tilde \a \tilde \om \xi_0)' a \end{array} \right], \]
and therefore
\be\label{fdl8}
B=\intl_{O(\ell)} d\a \intl_{O(k)}   f\left (\tilde \a \tilde \om v_0 \right ) \, d\om.\ee
Now we change the order of integration and set
\[ \gam =\a \om \in O(\ell), \qquad \tilde \gam= \left[\begin{array} {cc}
\gam & 0  \\ 0 &  \mathrm{I}_{n+1-\ell} \end{array} \right].\]
This gives
\be\label{fdl9x}
B=\intl_{O(\ell)}   f (\tilde \gam v_0) \, d\gam, \quad  v_0=
\left[\begin{array} {c}
e_1 \, \sqrt{1\!-\!|(\tilde \gam \xi_0)' a|^2} \\ (\tilde \gam \xi_0)' a \end{array} \right].\ee
Let
\be\label{fdl91}
(\tilde \gam \xi_0)' a=\left[\begin{array} {c}
b_1 \\ b_2 \end{array} \right],\ee
where $b_1$ and $b_2$ are  the $(\ell -k)$-vector and the $(n+1-\ell)$-vector, respectively.
Then the unit vector  $v_0$ can be represented in the form
\[
v_0=\left[\begin{array} {c}
\th\, \sqrt{1\!-\!|b_2|^2} \\ b_2 \end{array} \right]= \left[\begin{array} {cc}
\b & 0  \\ 0 &  \mathrm{I}_{n+1-\ell} \end{array} \right] \left[\begin{array} {c}
 e_1\, \sqrt{1\!-\!|b_2|^2} \\ b_2 \end{array} \right]
\]
with some $\th \in \bbs^{\ell-1}$ and $\b\in O(\ell)$.

To find a simple expression for $b_2$, we set $\tilde \eta_0 =\left[\begin{array} {c}
\tilde 0\\ \mathrm{I}_{n+1-\ell} \end{array} \right]$, where $\tilde 0$ is the zero matrix having
$\ell -k$ rows and $n+1-\ell$  columns. Then,  by (\ref{fdl91}),
\bea
b_2&=&\tilde \eta'_0 (\tilde \gam \xi_0)' a=(\xi_0 \tilde \eta_0)'\tilde \gam'a\nonumber\\
&=&
 \left (\left[\begin{array} {cc}
0 & 0  \\ \mathrm{I}_{\ell -k}  &  0 \\ 0 &  \mathrm{I}_{n+1-\ell} \end{array} \right]
\left[\begin{array} {c}
\tilde 0\\ \mathrm{I}_{n+1-\ell} \end{array} \right]\right )' \tilde \gam'a =\eta'_0\tilde \gam'a\nonumber\\
&=& (\tilde \gam \eta_0)'a=\eta'_0a. \nonumber\eea
Thus,

\[
v_0= \left[\begin{array} {cc}
\b & 0  \\ 0 &  \mathrm{I}_{n+1-\ell} \end{array} \right] \left[\begin{array} {c}
e_1\, \sqrt{1\!-\!|\eta'_0a|^2} \\\eta'_0a \end{array} \right].
\]
Substituting this expression in (\ref{fdl9x}) and changing variable $\gam\b =\del$, we obtain
\[
B=\intl_{O(\ell)}   f\left ( \tilde \del \left[\begin{array} {c}
 e_1\, \sqrt{1\!-\!|\eta'_0a|^2} \\ \eta'_0a \end{array} \right] \right ) \, d\del, \quad \tilde \del=\left[\begin{array} {cc}
\del & 0  \\ 0 &  \mathrm{I}_{n+1-\ell} \end{array} \right].\]
 The latter coincides with (\ref{fdl6}) up to notation.
\end{proof}

\begin {remark} An interested reader might be  advised to obtain an analogue of (\ref{fafaaa})  for the parallel slice transform $\Pi_a$.
\end{remark}

\bibliographystyle{amsplain}

\end{document}